\documentclass[12pt, amsfonts]{amsart}

\usepackage{amsmath,amssymb,amsthm}
\usepackage[top=3.5cm, bottom=3.5cm, left=2.5cm, right=2.5cm, includefoot]{geometry}
\usepackage[colorlinks=true]{hyperref}
\usepackage[all]{xy}

\numberwithin{equation}{section}

\SelectTips{eu}{12}

\newtheorem{theorem}{Theorem}[section]
\newtheorem{proposition}[theorem]{Proposition}
\newtheorem{lemma}[theorem]{Lemma}
\newtheorem{corollary}[theorem]{Corollary}

\theoremstyle{definition}
\newtheorem{definition}[theorem]{Definition}
\newtheorem{example}[theorem]{Example}

\theoremstyle{remark}
\newtheorem{remark}[theorem]{Remark}

\newcommand{\Z}{\mathbb{Z}}
\newcommand{\Q}{\mathbb{Q}}
\newcommand{\R}{\mathbb{R}}
\newcommand{\C}{\mathbb{C}}

\newcommand{\U}{\mathrm{U}}
\newcommand{\UFP}{\mathcal{U}}
\newcommand{\Ures}{\mathrm{U_{res}}}
\newcommand{\diag}{\mathrm{diag}}
\newcommand{\UH}{\mathrm{UH}}

\title{Hilbert bundles with ends}

\author{Tsuyoshi Kato}
\address{Department of Mathematics, Kyoto University, Kyoto, 606-8502, Japan}
\email{tkato@math.kyoto-u.ac.jp}

\author{Daisuke Kishimoto}
\address{Department of Mathematics, Kyoto University, Kyoto, 606-8502, Japan}
\email{kishi@math.kyoto-u.ac.jp}

\author{Mitsunobu Tsutaya}
\address{Faculty of Mathematics, Kyushu University, Fukuoka 819-0395, Japan}
\email{tsutaya@math.kyushu-u.ac.jp}

\subjclass[2010]{57R20, 57R22, 46L80, 46C05, 34K08}

\keywords{Hilbert bundle, end, unitary operator of finite propagation, uniform Roe algebra, pushforward of a vector bundle, spectral decomposition, Fourier transform, harmonic oscillator}

\begin{document}

  \maketitle

  \baselineskip.525cm

  \begin{abstract}
    Given a countable metric space, we can consider its end. Then a basis of a Hilbert space indexed by the metric space defines an end of the Hilbert space,
     which is a new notion and different from an end as a metric space.
    Such an indexed basis also defines unitary operators of finite propagation, and these operators preserve an end of a Hilbert space. Then, we can define a Hilbert bundle with end, which lightens up new structures of Hilbert bundles. In a special case, we can define characteristic classes of Hilbert bundles with ends, which are new invariants of Hilbert bundles. We show Hilbert bundles with ends appear in natural contexts. First, we generalize the pushforward of a vector bundle along a finite covering to an infinite covering, which is a Hilbert bundle with end under a mild condition. Then we compute characteristic classes of some pushforwards along infinite coverings. Next, we will show the spectral decompositions of nice differential operators give rise to Hilbert bundles with ends, which elucidate new features of spectral decompositions. The spectral decompositions we will consider are the
    Fourier transform and the harmonic oscillators.
  \end{abstract}

  \setcounter{tocdepth}{1}
  \tableofcontents

  \parskip .05in
  \parindent .0pt


  \section{Introduction}\label{Introduction}

  In this paper, we introduce a new structure of a Hilbert bundle, called an \emph{end}, which
    enables us to elucidate hidden topological features of a Hilbert bundle. We show
  that it appears in some natural contexts in mathematical physics. We define characteristic classes of a Hilbert bundle with end in an important special case, which are our basic tools for investigation. It is significant that our approach to  Hilbert bundles with ends is based on algebraic topology, though ends of spaces are studied standardly by analytic methods.


  \subsection{End of a Hilbert space}

  We start with defining an end of a Hilbert space. Let $X$ be a topological space. Given a compact set $K\subset X$, an \emph{end} of $X$ with respect to $K$ is a suitable equivalence class of descending sequences $W\supset U_1\supset U_2\supset\cdots$ such that
  \[
    \bigcap_{i=1}^\infty\mathrm{cl}(U_i)=\emptyset
  \]
  where $W$ is a neighborhood of $X-K$. Then the
  end of a CW complex captures a particular behavior at infinity, or out of a compact set, of a space. There is an algebraic analogy such that an end of a chain complex can be defined. See \cite{HR} for a comprehensive study of ends of CW complexes and chain complexes.

  What is an end of a Hilbert space? Trivially, we can define an end of a Hilbert space as a topological space. However, this definition only respects topology and does not respect structures of a Hilbert space, especially linearity. Hence,
   an end of a Hilbert space should be defined differently. We observe a simple example to get hold of a definition of an end of a Hilbert space. Let $L^2(S^1)$ denote the Hilbert space of square integrable complex valued functions on $S^1$, where $S^1$ is understood as the unit circle in $\C$. The Hilbert space $L^2(S^1)$ has orthonormal basis $\{z^n\}_{n\in\Z}$, and asymptotic features of a function $f(z)=\sum_{n\in\Z}a_nz^n\in L^2(S^1)$ are described by the behavior of higher terms $\sum_{|n|\ge N}a_nz^n$ as $N$ goes to $\infty$. Then a sequence of subspaces of $L^2(S^1)$
  \[
    H_N\supset H_{N+1}\supset H_{N+2}\supset\cdots
  \]
  can be thought of as an end of $H$ with respect to the basis $\{z^n\}_{n\in\Z}$, where $H_M$ is the completion of a subspace spanned by $z^n$ for $|n|\ge M$. Summarizing, an end of the Hilbert space $L^2(S^1)$ is defined by a particular basis $\{z^n\}_{n\in\Z}$,
  and the ``shape" of the end is determined by an end of the index set $\Z$ of the basis.

  With the above observation, we proceed to define an end of a Hilbert space. Let $H$ be a separable Hilbert space, and let $I$ be a countable metric space with base point $i_0$. Given an orthonormal basis $\{e_i\}_{i\in I}$ indexed by  $I$, we set $H_n$ be the completion of a subspace of $H$ spanned by $e_i$ with $d(i,i_0)\ge n$. Then we get a sequence of subspaces
  \[
    H=H_0\supset H_1\supset H_2\supset\cdots
  \]
  which is an end of $H$ with respect to the basis $\{e_i\}_{i\in I}$. Observe that an end of a pointed countable metric space $I$ determines the "type" of an end of $H$, and a basis $\{e_i\}_{i\in I}$ realizes the end type determined by $I$ in $H$. An end of a Hilbert space is a new notion, which appears in some important situations such as Fourier analysis as above. Note that by definition, an end of a Hilbert space is different from an end as a topological space, as mentioned above.

  Suppose we are given two orthonormal bases of $H$ which are indexed by the same pointed countable metric space $I$. When are the corresponding ends equivalent? Let $E=\{H_n\}_{n\ge 0}$ and $E'=\{H_n'\}_{n\ge 0}$ be ends of $H$ corresponding to given two orthonormal bases indexed by $I$. Since ends of a Hilbert space $H$ should represent its behavior at infinity, the ends $E$ and $E'$ should be equivalent if they differ only by a finite fluctuation. Then, we define that two ends $E$ and $E'$ are equivalent if there is $L>0$ such that for each $n$, the inclusions hold:
  \[
    H_{n-L}\subset H'_n\subset H_{n+L}.
  \]
  Note that two ends are equivalent if their filters satisfy the above inclusions for large $n$  so that the basepoint of $I$ does not matter. The   equivalence of ends can be representerd by unitary operators of finite propagation. Primitive versions of such operators appear in several contexts of analysis and geometry \cite{CGT,C,GNVW,R}. We define unitary operators of finite propagation in a quite general setting, which broadens up possible directions of study.

  \begin{definition}
    Let $H$ be a separable Hilbert space equipped with a specific basis $\{e_i\}_{i\in I}$ indexed by a countable metric space $I$. A unitary operator $U$ on $H$ is of \emph{finite propagation} if there is $L>0$ such that
    \[
      U_{ij}=0\quad\text{for}\quad d(i,j)>L
    \]
    where $U=(U_{ij})_{i,j\in I}$ is the matrix description defined by the basis $\{e_i\}_{i\in I}$.
  \end{definition}

  Let $\{e_i\}_{i\in I}$ and $\{e'_i\}_{i\in I}$ be bases of $H$ indexed by a pointed countable metric space $I$. Then the corresponding ends are equivalent if and only if there is a unitary operator $U$ of finite propagation with respect to $\{e_i\}_{i\in I}$ such that for each $i$,
  \[
    Ue_i=e'_i.
  \]


  \subsection{Hibert bundle with end}

  We have defined an end of a Hilbert space and its automorphisms as above. We next define a Hilbert bundle with end. There are several inequivalent notions of Hilbert bundles as in \cite{D}. In this paper, a Hilbert bundle
  means a map $E\to B$ such that each fiber is isometric with a Hilbert space $H$, it is locally trivial, and the transition functions are continuous maps into $\U(H)$, the space of unitary operators on $H$ with the operator norm topology. By the classical result of Kuiper \cite{K}, the classifying space $B\U(H)$ is contractible, implying every Hilbert bundle over a CW complex must be trivial. Hence, there is no topological feature of Hilbert bundles over CW complexes. In order to understand hidden topological features of a Hilbert bundle, we need to consider an additional structure of a bundle. It is standard to globalize a structure of a Hilbert space to an additional Hilbert bundle. A typical example is a polarization of a Hilbert bundle (see \cite{PS}). Here, we consider an end of a Hilbert space to be globalized to an additional structure of a Hilbert bundle.

  Let $H$ be a Hilbert space equipped with a specific basis indexed by a countable metric space, and let $\UFP(H)$ denote the group of unitary operators of finite propagations on $H$. We will equip the group $\UFP(H)$ with the inductive limit topology with respect to propagation (See Section \ref{Unitary operators of finite propagation} for details), instead of the operator norm topology. An advantage of the weak topology can be seen in a connection to the corresponding uniform Roe algebra, which will be precisely discussed in Section \ref{Unitary operators of finite propagation}. Now we define a Hilbert bundle with end.

  \begin{definition}
    Let $H$ be a Hilbert space equipped with a specific basis indexed by a countable metric space. We say that a Hilbert bundle with fiber $H$ has an \emph{end} if its transition functions are continuous maps into $\UFP(H)$.
  \end{definition}

  As Hilbert bundles with ends are defined in terms of unitary operators of finite propagation, we need to know their properties, especially topological ones, in order to investigate Hilbert bundles with ends. In the previous work \cite{KKT1,KKT2}, the authors studied the topological group of unitary operators of finite propagation on $\ell^2(\Z^n,\C)$, denoted by $\UFP(n)$, and the homotopy type of $\UFP(1)$ is explicitly described in \cite{KKT1}. In this paper, we will further prove the Bott peroidicity of $\UFP(1)$ by passing through the uniform Roe algebra, so that the homotopy type of $B\UFP(1)$ is explicitly described.
  In particular, we can determine the rational cohomology of $B\UFP(1)$ as follows. Let $\mathcal{B}$ be a basis of a $\Q$-vector space $\ell^\infty(\Z,\Z)_S^\vee$ of uncountable dimension (see Section \ref{Unitary operators of finite propagation} for the definition of this vector space).

  \begin{theorem}
    [Corollary \ref{cohomology U(1)}]
    The rational cohomology of $B\UFP(1)$ is given by
    \[
      H^*(B\UFP(1);\Q)=\Q[\alpha_n(b)\mid b\in\mathcal{B}]\otimes\Lambda(\beta_n\mid n\ge 1)
    \]
    where $|\alpha_n(b)| = 2n$ and $|\beta_n|=2n-1$.
  \end{theorem}

  Then, for a Hilbert bundle with end $E\to B$, we can define characteristic classes
  \[
    \alpha_n(E;a)\quad\text{and}\quad\beta_n(E)
  \]
  which are independent from the choice of a basis $\mathcal{B}$, where $a\in\ell^\infty(\Z,\Z)_S^\vee$. These characteristic classes will be the main tool for investigating Hilbert bundles with ends modeled on $\ell^2(\Z,\C)$. The even dimensional characteristic classes $\alpha_n(E;\alpha)$ are fairly new. The odd dimensional classes $\beta_n(E)$ actually come from the restricted unitary group (see \cite{PS}), and this is a new approach to interpret the restricted unitary group in terms of ends of Hilbert bundles.

  Based on  the above general theory, we will show Hilbert bundles with ends appear in several important contexts in mathematical physics. First, we will consider the \emph{pushforward} of a finite-dimensional vector space along an infinite covering. To begin with, we will consider the fiberwise completion of a direct sum of infinitely
  many finite-dimensional vector bundles, called a \emph{completed sum}, which corresponds to the pushforward along an infinite-trivial covering. We will show that under a mild condition, the completed sum is a Hilbert bundle with end (Theorem \ref{completed sum end}). Then, we will specialize to the case of $\ell^2(\Z,\C)$ and prove:

  \begin{proposition}
    [Proposition \ref{completed sum c_n}]
    Given any $a\in\ell^\infty(\Z,\Z)_S^\vee$ and $n\ge 1$, there is a completed sum $E\to B$, which is an $\ell^2(\Z,\C)$-bundle with end, such that $\alpha_n(E;a)\ne 0$.
  \end{proposition}

  We will also show interesting properties of Hilbert bundles with ends considered in this proposition with respect a vector $a\in\ell^\infty(\Z,\Z)_S^\vee$ defining the even dimensional characteristic class (Proposition \ref{E(a) property}).

  Recall that the pushforward (or the direct image bundle) of a finite-dimensional vector bundle is defined for a finite covering, where finiteness of a covering is demanded if we focus on a class of finite-dimensional vector bundles. We will define the pushforward of a finite-dimensional vector bundles along a countable covering as a natural generalization of the pushforward along a finite covering, so that a completed sum is the pushforward along a trivial covering. We will show that under a mild condition, the pushforward along a countable covering is a Hilbert bundle with end (Theorem \ref{pushforward end}). In particular, the pushforward along an infinite cyclic covering  turns out to be a Hilbert bundle with end modeled on $\ell^2(\Z,\C)$, under a mild condition. Then we  specialize to the case of
  an infinite cyclic coverings, and show properties of pushforwards, analogous to the classical case, in terms of characteristic classes. Furthermore, we will construct the pushforward with a non-trivial 1-dimensional characteristic class $\beta_1$.

  \begin{proposition}
    [Corollary \ref{beta_1}]
    Let $\pi\colon X\to Y$ be a normal infinite cyclic covering satisfying the conditions pushforward end in Theorem \ref{pushforward end}. Then the pushforward of the trivial bundle $X\times\C$ along $\pi$ is a Hilbert bundle with end such that $\beta_1\ne0$.
  \end{proposition}

  We have not found a pushforward with non-trivial $\beta_n$ for $n\ge 2$. It would be interesting to find such a pushforward.


  Next, we will consider the spectral decomposition of a differential operator. We will construct Hilbert bundles with ends from two sorts of spectral decompositions, the Fourier transform and the harmonic oscillators. As in the above observation on $L^2(S^1)$, the spectral decomposition given by the Fourier transform provides an orthonormal basis. Using the orthonormal basis, we will construct a natural family of orthonormal basis parametrized by $S^1$, and we will get:

  \begin{theorem}
    [Theorem \ref{Fourier transform bundle}]
   The  Fourier transform on $S^1$ yields a Hilbert bundle with end $q\colon\mathcal{L}^1\to S^1$ modeled on $\ell^2(\Z,\C)$ such that $\beta_1(\mathcal{L}^1)\ne 0$.
  \end{theorem}
We will discuss higher dimensional analogue of this construction.

We consider  $2$-dimensional harmonic oscillators. We begin with analyzing a hidden parameter of the
$2$-dimensional harmonic oscillators such that the parameter space is homeomorphic with the disjoint union of $S^2$ and two points. Then we will get a family spectral decomopsitions of the $2$-dimensional harmonic oscillators parametrized over $S^2$. We prove that
this family yields a Hilbert bundle with end.

  \begin{theorem}
    [Theorem \ref{2-dim}]
    The spectral decomposition of a family of $2$-dimensional harmonic oscillators yields a Hilbert bundle with end $\mathcal{E}\to S^2$ modeled on $\ell^2(\Z,\C)$, such that $\alpha_1(\mathcal{E};a)\ne 0$ for certain $a\in\ell^\infty(\Z,\Z)_S^\vee$.
  \end{theorem}

  Our construction of Hilbert bundles with ends from the spectral decompositions of differential operators grew out of an attempt to find a new invariant of moduli spaces attached a Hilbert space at each point, which come out of certain differential operators. A new aspect of our construction is that the corresponding Hilbert bundles with ends reflect the behavior of eigenspaces with high eigenvalues. In fact, Fredholm theory concerns spectral behaviors only near zero. Hence, we can say that our construction will lighten up a new features of spectral decompositions, which  is a kind of spectral topology. There is a technical issue to try to treat a general self-adjoint elliptic operator on a compact manifold, since multiplicities of eigenvalues may be unbounded, in which case it cannot be straightforwardly treated within the class of unitary operators of finite propagation. Then the general self-adjoint elliptic operator case needs more elaborations.


  \subsection{Organization}

  We explain the organization of the paper. Section \ref{Unitary operators of finite propagation} collects previous results and proves properties of unitary operators of finite propagation. We will especially consider the topological group of unitary operators of finite propagation on $\ell^2(\Z,\C)$, which was studied in \cite{KKT1,KKT2}. We will prove its Bott periodicity and determine the rational cohomology of the classifying space. This enables us to define characteristic classes for $\ell^2(\Z,\C)$-bundles with ends. Section \ref{The abelian group} shows some properties of the coinvariant group of $\ell^\infty(\Z,\Z)$ under the action of the shift operator. This abelian group appears in the homotopy groups of $\UFP(1)$, and its properties help understanding features of $\ell^2(\Z,\C)$-bundles with ends. Section \ref{Pushforward} studies pushforwards of vector bundles along infinite coverings. We will start with the trivial bundle case, and then pass to the general case. We will prove the pushforward of a vector bundle along a countable covering is a Hilbert bundle with end under a mild condition, and show its basic properties including characteristic classes. Section \ref{Fourier transform} constructs a Hilbert bundle with end by using the Fourier transform. The resulting bundle will be shown to have a non-trivial 1-dimensional characteristic class. We also discuss a higher dimensional generalization. Section \ref{Harmonic oscillator} constructs a Hilbert bundle with end from a 2-dimensional harmonic oscillator. The construction is based on the spectral decomposition as well as the Fourier transform. The resulting bundle will be shown to have a non-trivial 2-dimensional characteristic class.

  \subsection*{Acknoledgement}

  TK was supported by JSPS KAKENHI 17K18725 and 17H06461, DK was supported by JSPS KAKENHI 17K05248 and 19K03473, and MT was supported by JSPS KAKENHI 19K14535.


  \section{Unitary operators of finite propagation}\label{Unitary operators of finite propagation}

  This section defines precisely the topological group of unitary operators of finite propagation. Then we recall from \cite{KKT1} results on the topological group of unitary operators of finite propagation on $\ell^2(\Z,\C)$, denoted by $\UFP(1)$, and its subgroup of periodic operators. We also recall from \cite{KKT2} results on the completed version of $\UFP(1)$, and prove $\UFP(1)$ and its completed version are of the same homotopy type, which is implicit in \cite{KKT2}. This enables us to prove the Bott periodicity of $\UFP(1)$ and determine the rational cohomology of the classifying space $B\UFP(1)$. Then we will define characteristic classes of $\ell^2(\Z,\C)$-bundles with ends.


  \subsection{Topology}

  Let $H$ be a Hilbert space equipped with a specific basis indexed by a countable metric space $I$. The \emph{propagation} of a unitary operator $U$ on $H$ is defined by
  \[
    \mathrm{prop}(U)=\sup\{d(i,j)\mid U_{ij}\ne 0\}
  \]
  where the matrix description $U=\{U_{ij}\}_{i,j\in I}$ is defined by the given basis of $H$. Let $\UFP_L(H)$ be the set of all unitary operators $U$ with $\mathrm{prop}(U)\le L$. Then
  \begin{equation}
    \label{UFP colim}
    \UFP(H)=\lim_{L\to\infty}\UFP_L(H)
  \end{equation}
  as a set. By definition, we have $\mathrm{prop}(UV)\le\mathrm{prop}(U)+\mathrm{prop}(V)$, so that $\UFP_L(H)$ need not be closed under product. However, $\UFP(H)$ is closed under product.

  We will equip $\UFP_L(H)$ the topology given by operator norm. Then we will equip $\UFP(H)$ with the inductive limit topology by \eqref{UFP colim}. Namely, $C\subset\UFP(H)$ is closed if and only if $C\cap \UFP_L(H)$ is closed in $\UFP_L(H)$ for each $L$. Notice that the inductive limit topology of $\UFP(H)$ may be different from the operator norm topology. There are two reasons for this choice of topology of $\UFP(H)$. One reason is that it suites to study homotopical properties, and the other reason is that it is compatible with taking the uniform Roe algebra, as is seen below.


  \subsection{Homotopy type}

  Let $\ell^2(\Z^n,\C)$ denote the Hilbert space of square summable functions from $\Z^n$ to $\C$. For $\mathbf{k}\in\Z^n$, let $e(\mathbf{k})\in\ell^2(\Z^n,\C)$ be a function sending $\mathbf{k}$ to 1 and other elements of $\Z^n$ to 0. Then $e(\mathbf{k})$ form an orthonormal basis of $\ell^2(\Z^n,\C)$ indexed by $\Z^n$, where $\Z^n$ is obviously a countable metric space. We will always consider this orthonormal basis for $\ell^2(\Z^n,\C)$, so that we will consider unitary operators of finite propagation on $\ell^2(\Z^n,\C)$ with respect to this orthonormal basis. We will use an abbreviated notation
  \[
    \UFP(n)=\UFP(\ell^2(\Z^n,\C)).
  \]
  We recall properties of $\UFP(1)$ proved in \cite{KKT1}.

  In \cite{KKT1}, the authors determined the homotopy type of $\UFP(1)$. To state the result, we set notation. Let $\ell^\infty(\Z,\Z)$ denote the abelian group of bounded integer sequences indexed by $\Z$. Define
  \[
    S\colon\ell^\infty(\Z,\Z)\to\ell^\infty(\Z,\Z),\quad(a_i)\mapsto(b_i),\quad\text{where}\quad b_i=a_{i+1}
  \]
  which we call the \emph{shift} operator. Let $\ell^\infty(\Z,\Z)_S$ denote the abelian group of coinvariants of $S$, that is,
  \[
    \ell^\infty(\Z,\Z)_S=\ell^\infty(\Z,\Z)/\{(1-S)(a)\mid a\in\ell^\infty(\Z,\Z)\}.
  \]
  The abelian group $\ell^\infty(\Z,\Z)_S$ is important to describe the homotopy type of $\UFP(1)$ and to examine Hilbert bundles with ends as below. Here, we give the most important property of $\ell^\infty(\Z,\Z)_S$ proved in \cite[Proposition 5.3]{KKT1}, and other properties will be discussed in Section \ref{The abelian group}.

  \begin{proposition}
    \label{Q-vector space}
    The abelian group $\ell^\infty(\Z,\Z)_S$ is a $\Q$-vector space.
  \end{proposition}

  Let $\U(n)$ denote the group of unitary matrices of rank $n$, and let $\U(\infty)=\lim_{n\to\infty}\U(n)$. For pointed spaces $\{X_n\}_{n\ge 1}$, we define their weak product by
  \[
    \overset{\circ}{\prod_{n\ge 1}}X_n=\lim_{n\to\infty}\prod_{k=1}^nX_k.
  \]
  Now we state the result on the homotopy type of $\UFP(1)$ proved in \cite[Theorem 1.2]{KKT1}.

  \begin{theorem}
    \label{homotopy type UFP}
    There is a homotopy equivalence
    \[
      \UFP(1)\simeq\Z\times B\U(\infty)\times\overset{\circ}{\prod_{n\ge 1}}K(\ell^\infty(\Z,\Z)_S,2n-1).
    \]

  \end{theorem}

  \begin{corollary}
    \label{homotopy group UFP}
    There is an isomorphism
    \[
      \pi_*(\UFP(1))\cong\begin{cases}\Z&*\text{ is even}\\\ell^\infty(\Z,\Z)_S&*\text{ is odd.}\end{cases}
    \]
  \end{corollary}

  We explain an interpretation of the homotopy groups of $\UFP(1)$. A generator of $\UFP(1)$, or more general groups of unitary operators of finite propagation, is given in \cite{GNVW}.

  \begin{proposition}
    \label{pi_0}
    A generator of $\pi_0(\UFP(1))\cong\Z$ is represented by the shift operator $S$.
  \end{proposition}

  Let $\Ures$ denote the group of restricted unitary operators on $\ell^2(\Z,\C)$, where $\ell^2(\Z,\C)$ is given the canonical polarization. Then as in \cite{PS}, there is a homotopy equivalence
  \[
    \Ures\simeq\Z\times B\U(\infty).
  \]
  In particular, the homotopy groups of $\Ures$ are given by
  \[
    \pi_*(\Ures)\cong\begin{cases}\Z&*\text{ is even}\\0&*\text{ is odd.}\end{cases}
  \]
  The following will be proved later.

  \begin{proposition}
    \label{inclusion Ures}
    The inclusion $\UFP(1)\to\Ures$ induces an isomorphism in even homotopy groups.
  \end{proposition}

  By this proposition, we can say that the even homotopy groups of $\UFP(1)$ give information about ends of Hilbert bundles concerning polarization, and the odd homotopy groups of $\UFP(1))$ are fairly new invariants of Hilbert bundles. However, there has not been such an interpretation of the homotopy groups of $\Ures$, so our approach gives a new insight on the topology of $\Ures$ too. Finally, we recall from \cite{GNVW} that $\pi_0(\UFP(1))\cong\Z$ is generated by the shift operator
  \[
    S\colon\ell^2(\Z,\C)\to\ell^2(\Z,\C),\quad(S(a_i))_j=a_{j+1}
  \]
  which is obviously a unitary operator of finite propagation.


  \subsection{Periodic operators}

  In \cite{KKT1}, a topological subgroup of $\UFP(1)$ attaining an interesting part of the homotopy groups of $\UFP(1)$ was introduced. We recall it here. An element $U\in\UFP(1)$ is \emph{periodic} if
  \[
    S^nUS^{-n}=U
  \]
  for some non-zero integer $n \ne 0$. Clearly, a product of two periodic unitary operators of finite propagation is periodic. Let $\widehat{\UFP}$ denote the topological subgroup of $\UFP(1)$ consisting of periodic elements.

  \begin{definition}
    We say that an $\ell^2(\Z,\C)$-bundle has a \emph{periodic end} if its transition functions are continuous functions into $\widehat{\UFP}$.
  \end{definition}

  We recall results on $\widehat{\UFP}$ proved in \cite{KKT1}. As well as Theorem \ref{homotopy type UFP}, the homotopy type of $\widehat{\UFP}$ was described.

  \begin{theorem}
    \label{homotopy type periodic}
    There is a homotopy equivalence
    \[
      \widehat{\UFP}\simeq\Z\times B\U(\infty)\times\overset{\circ}{\prod_{n\ge 1}}K(\Q,2n-1).
    \]
  \end{theorem}

  \begin{corollary}
    \label{homotopy group periodic}
    There is an isomorphism
    \[
      \pi_*(\widehat{\UFP})\cong\begin{cases}\Q&*\text{ is odd}\\\Z&*\text{ is even.}\end{cases}
    \]
  \end{corollary}

  By the construction of homotopy equivalences of Theorems \ref{homotopy type UFP} and \ref{homotopy type periodic}, we have:

  \begin{corollary}
    \label{inclusion periodic}
    The inclusion $\widehat{\UFP}\to\UFP(1)$ is an isomorphism in even homotopy groups and identified with the injection $\Q\to\ell^\infty(\Z,\Z)_S$ which is the extension of the inclusion
    \[
      \Z\to\ell^\infty(\Z,\Z)_S,\quad a\mapsto[(\ldots,a,a,a,\ldots)]
    \]
    in odd homotopy groups.
  \end{corollary}


  \subsection{Uniform Roe algebra}

  In \cite{KKT1}, the homotopy equivalence of Theorem \ref{homotopy type UFP} was constructed by using particular properties of $\ell^2(\Z,\C)$, so that the method is not applicable to $\ell^2(\Z^n,\C)$ for $n>1$. However, if we consider a more analytical side, then we can use rich techniques of $K$-theory. This is the motivation to consider uniform Roe algebras. See \cite{KKT2} for details.

  First, we recall the definition of the uniform Roe algebras with respect to operators of finite propagation.

  \begin{definition}
    The norm closure of the space of finite propagation operators in the bounded operators on $\ell^2(\Z^n,\C)$ is called the uniform Roe algebra, and denoted by $C^\ast_u(|\Z^n|)$.
  \end{definition}

  Let $\U(C^*_u(|\Z^n|))$ denote the subspace of $C^*_u(|\Z^n|)$ consisting of unitary elements. As mentioned above, we can employ $K$-theory techniques to investigate $\U(C^*_u(|\Z^n|))$, which is possible because spaces are completed. In particular, by using such techniques, the authors \cite{KKT2} described the homotopy type of $\U(C^*_u(|\Z^n|))$.

  \begin{theorem}
    \label{homotopy type Roe}
    There exists a homotopy equivalence
    \[
      \U(C^*_u(|\Z|))\simeq\Z\times B\U(\infty)\times\prod_{n\ge1}^\circ K(\ell^\infty(\Z,\Z)_S,2n-1).
    \]
  \end{theorem}

  Moreover, the Bott periodicity for $\U(C^*_u(|\Z|))$ was proved in \cite{KKT2}.

  \begin{theorem}
    \label{Bott periodicity}
    There is a homotopy equivalence
    \[
      B\U(C^*_u(|\Z|))\simeq\Omega\U(C^*_u(|\Z|)).
    \]
  \end{theorem}

  The homotopy type of $\U(C^*_u(|\Z^2|))$ is also studied in \cite{KKT2}, and we have:

  \begin{theorem}
    \label{homotopy type UFP 2}
    There exists a homotopy equivalence
    \[
      \U(C^*_u(|\Z^2|))\simeq V_1\times\U(\infty)\times\prod_{n\ge1}^\circ K(V_0,2n-1)\times K(V_1,2n)
    \]
    where $V_0$ is the kernel of the natural map $K_0(C^*_u(|\Z^2|))\to K_0(C^*(|\Z^2|))$ and $V_1=K_1(C^*_u(|\Z^2|))$.
  \end{theorem}

  The abelian groups $V_0$ and $V_1$ are not explicitly determined yet. In \cite{KKT2}, the homotopy groups of $\U(C^*_u(|\Z^n|))$ are also determined such that
  \[
    \pi_*(\U(C^*_u(|\Z^n|)))\cong K_*(\U(C^*_u(|\Z^n|)))
  \]
  where $K_*(\U(C^*_u(|\Z^n|)))$ are not yet determined either. See \cite{KKT2} for details.


  \subsection{Bott periodicity and characteristic classes}

  We prove the Bott periodicity of $\UFP(1)$ and determine the rational cohomology of the classifying space $B\UFP (1)$ by applying results on $\U(C^*_u(|\Z|))$. By Theorems \ref{homotopy type UFP} and \ref{homotopy type Roe}, we can naively expect that the inclusion $\UFP(1)\to\U(C^*_u(|\Z|))$ is a homotopy equivalence. We prove this expectation is true.

  \begin{lemma}
    \label{inclusion odd}
    The inclusion $\UFP(1)\to\U(C^*_u(|\Z|))$ is an isomorphism in odd homotopy groups.
  \end{lemma}

  \begin{proof}
    Let $\ell^\infty(\Z,\C)$ be the set of
    bounded sequences of complex numbers indexed by $\mathbb{Z}$, and let $\U_L$ denote the space of $L\times L$ unitary matrices with entries in $\ell^\infty(\Z,\C)$.
    More generally, for an algebra $A$, let
    $U_L(A)$ be the space of $L\times L$ unitary matrices with entries in $A$.
    As in \cite[Proof of Theorem 1.1]{KKT1}, there is a block diagonal embedding $\phi\colon\U_L\to\UFP(1)$, so that for $1\le i\le L$, the induced map \[
      \phi_*\colon\pi_{2i-1}(\U_L)\to\pi_{2i-1}(\UFP(1))
    \]
    is identified with the quotient map $\ell^\infty(\Z,\Z)\to\ell^\infty(\Z,\Z)_S$.

    Define a map $g\colon\U_1(C^*_u(|\Z|))\to\U_L(C^*_u(|\Z|))$ by $(u_{ij})_{i,j\in\Z}$ by $g((u_{ij})_{i,j\in\Z})=(U_{ij})_{i,j\in\Z}$ where
    \[
      U_{ij}=
      \begin{pmatrix}
        u_{Li,Lj}&u_{Li,Lj+1}&\cdots&u_{Li,Lj+L-1}\\
        u_{Li+1,Lj}&u_{Li+1,Lj+1}&\cdots&u_{Li+1,Lj+L-1}\\
        \vdots&\vdots&&\vdots\\
        u_{Li+L-1,Lj}&u_{Li+L-1,Lj+1}&\cdots&u_{Li+L-1,Lj+L-1}
      \end{pmatrix}.
    \]
    Then the composition
    \[
      \bar{\phi}\colon\U_L\xrightarrow{\phi}\UFP(1)\xrightarrow{\iota}\U_1(C^*_u(|\Z|))\xrightarrow{g}\U_L(C^*_u(|\Z|))
    \]
    coincides with the map $\U_L\to\U_L(C^*_u(|\Z|))$ induced from the diagonal inclusion, where $\iota$ denotes the inclusion. By \cite[Proposition 3.2]{KKT2} and \cite[Proof of Proposition 6.1]{KKT2}, the induced map of this composition
    \[
      \bar{\phi}_*\colon\pi_{2i-1}(\U_L)\to\pi_{2i-1}(\U_L(C^*_u(|\Z|)))
    \]
    for $1\le i\le L$ is also identified with the quotient map $\ell^\infty(\Z,\Z)\to\ell^\infty(\Z,\Z)_S$. Thus we obtain that the map $\iota\colon\pi_{2i-1}(\UFP(1))\to\pi_{2i-1}(\U_1(C^*_u(|\Z|)))$ is an isomorphism for $1\le i\le L$. The proof is complete as we let $L\to\infty$.
  \end{proof}

  \begin{lemma}
    \label{inclusion even}
    The inclusion $\UFP(1)\to\U(C^*_u(|\Z|))$ is an isomorphism in even homotopy groups.
  \end{lemma}

  \begin{proof}
    Let $G$ be the space of unitary operators $U=(U_{ij})_{i,j\in\Z}$ such that
     $U_{ij}=0$ if $(i,j) \ne (0,0)$
     and the multiple $i \cdot j \leq 0$ holds.
     Namely
     $i \geq 0$ and $j <0$, or $i <0$ and $j \geq 0$ holds.
     Then $G$ is a subgroup of $\Ures$ and contractible by Kuiper's theorem \cite{K}. In \cite[Proof of Theorem 1.2]{KKT1}, it is shown that there is a homotopy commutative diagram
    \[
      \xymatrix{
        \UFP(1) \ar[rr] \ar[d]
        & & Z\times B\U(\infty) \ar[d]^-{\simeq}\\
        \U(C^*_u(|\Z|)) \ar[r]
        & \Ures \ar[r]^-\simeq
        & \Ures/G
      }
    \]
    By \cite[Lemma 5.5]{KKT1} and \cite[Lemma 6.2]{KKT2}, the horizontal arrows are isomorphisms in even homotopy groups, completing the proof.
  \end{proof}

  Now we are ready to prove:

  \begin{theorem}
    \label{inclusion homotopy equiv}
    The inclusion $\UFP(1)\to\U(C^*_u(|\Z|))$ is a homotopy equivalence.
  \end{theorem}

  \begin{proof}
    Both $\UFP(1)$ and $\U(C^*_u(|\Z|))$ are homotopy equivalent to CW complexes.   These facts follow from
    Theorem \ref{homotopy type UFP} and
    Theorem \ref{homotopy type Roe}.
    Hence, the proof is done by Lemmas \ref{inclusion odd} and \ref{inclusion even} together with the J.H.C. Whitehead theorem.
  \end{proof}

  We prove Proposition \ref{inclusion Ures} as a corollary of Theorem \ref{inclusion homotopy equiv}.

  \begin{proof}
    [Proof of Proposition \ref{inclusion Ures}]
    Clearly, the inclusion $\UFP(1)\to\Ures$ factors as the composition of inclusions $\UFP(1)\to\U(C^*_u(|\Z|))\to\Ures$. By \cite[Lemma 6.2]{KKT2}, the inclusion $\U(C^*_u(|\Z|))\to\Ures$ is an isomorphism in even homotopy groups. Then the proof is done by Theorem \ref{inclusion homotopy equiv}.
  \end{proof}

  We are ready to prove the Bott periodicity of $\UFP(1)$.

  \begin{theorem}
    \label{Bott periodicity UFP}
    There is a homotopy equivalence
    \[
      B\UFP(1)\simeq\Omega\UFP(1).
    \]
  \end{theorem}

  \begin{proof}
    Since the inclusion $\UFP(1)\to\U(C^*_u(|\Z|))$ is a homomorphism, it induces a map between classifying spaces $B\UFP(1)\to B\U(C^*_u(|\Z|))$. By Theorem \ref{inclusion homotopy equiv}, the induced map is an isomorphism in homotopy groups, hence a homotopy equivalence by the J.H.C. Whitehead theorem. Thus by Theorem \ref{Bott periodicity}
    \[
      B\UFP(1)\simeq B\U(C^*_u(|\Z|))\simeq\Omega\U(C^*_u(|\Z|))\simeq\Omega\UFP(1)
    \]
    and the proof is complete.
  \end{proof}

  By Theorems \ref{homotopy type UFP} and \ref{Bott periodicity UFP} together with the Bott periodicity of $\U(\infty)$, we obtain:

  \begin{corollary}
    There is a homotopy equivalence
    \[
      B\UFP(1)\simeq\U(\infty)\times\overset{\circ}{\prod_{n\ge 1}}K(\ell^\infty(\Z,\Z)_S,2n).
    \]
  \end{corollary}

  By this corollary, we can determine the rational cohomology of the classifying space $B\UFP(1)$. Let $\ell^\infty(\Z,\Z)_S^\vee=\mathrm{Hom}(\ell^\infty(\Z,\Z)_S,\Q)$, and take a basis $\mathcal{B}$ of it.

  \begin{corollary}
    \label{cohomology U(1)}
    The rational cohomology of $B\UFP(1)$ is given by
    \[
      H^*(B\UFP(1);\Q)=\Q[\alpha_n(b)\mid b\in\mathcal{B},\,n\ge 1]\otimes\Lambda(\beta_{n}\mid n\ge 1)
    \]
    where $|\alpha_n(b)|=2n$ and $|\beta_n|=2n-1$.
  \end{corollary}

  Let $x(b)\in H^{2n}(B\UFP(1);\Q)$ be the dual of the Hurewicz image of $b\in\mathcal{B}\subset\ell^\infty(\Z,\Z)_S\cong\pi_{2n}(B\UFP(1))$. We choose a generator $\alpha_n(b)$ such that
  \[
    \alpha_n(b)=\frac{1}{(n-1)!}x(b).
  \]
  This is because the image of the Hurewicz map $\pi_{2n}(B\U(\infty))\to H_{2n}(B\U(\infty))$ is divisible by $(n-1)!$. See Corollary \ref{transfer c_n} for the convenience of this choice. We also choose a generator $\beta_n$ to be the suspension of $c_n$ in $H^{2n-1}(\U(\infty))\subset H^{2n-1}(B\UFP(1);\Q)$.

  Now we define characteristic classes of $\ell^2(\Z,\C)$-bundles with ends. We will see in Section \ref{Pushforward} that these characteristic classes give information about ends of Hilbert bundles.

  \begin{definition}
    Let $E\to X$ be an $\ell^2(\Z,\C)$-bundle with end, where $X$ is a CW complex, and let $a\in\ell^\infty(\Z,\Z)_S^\vee$ such that $a=a_1b_1+\cdots+a_kb_k$ for $a_i\in\Q$ and $b_i\in\mathcal{B}$. We define
    \[
      \alpha_n(E;a)=a_1f^*(\alpha_n(b_1))+\cdots+a_kf^*(\alpha_n(b_k))\quad\text{and}\quad \beta_n(E)=f^*(\beta_n)\in H^*(X;\Q)
    \]
    where $f\colon X\to B\UFP(1)$ is a classifying map of $E$.
  \end{definition}

  Remarks on the characteristic classes $\alpha_n(E;a)$ and $\beta_n(E)$ are in order. The characteristic classes $\alpha_n(E;a)$ are new invariants of $\ell^2(\Z,\C)$-bundles, and independent from the choice of basis $\mathcal{B}$ of $\ell^\infty(\Z,\Z)_S^\vee$. By Proposition \ref{inclusion Ures}, the characteristic classes $\beta_n(E)$ comes from $B\Ures$, so they give information about ends of $\ell^2(\Z,\C)$-bundle $E$ related to the polarization of $\ell^2(\Z,\C)$.

  We show basic properties of characteristic classes of $\ell^2(\Z,\C)$-bundles with ends, which are immediate from the definition.

  \begin{proposition}
    \label{characteristic class}
    Let $E\to X$ be an $\ell^2(\Z,\C)$-bundle with end.
    \begin{enumerate}
      \item For a map $f\colon Y\to X$ and $a\in\ell^\infty(\Z,\Z)_S^\vee$,
      \[
        \alpha_n(f^{-1}E;a)=f^*(\alpha_n(E;a))\quad\text{and}\quad\beta_n(f^{-1}E)=f^*(\beta_n(E)).
      \]

      \item For $a_1,a_2\in\ell^\infty(\Z,\Z)_S^\vee$,
      \[
        \alpha_n(E;a_1+a_2)=\alpha_n(E;a_1)+\alpha_n(E;a_2).
      \]
    \end{enumerate}
  \end{proposition}

  We also define characteristic classes for $\ell^2(\Z,\C)$-bundles with periodic ends. Let
  \[
    \hat{\alpha}_n=i^*(\alpha_n(a))\quad\text{and}\quad\hat{\beta}_n=y^*(\beta_n)
  \]
  where $i\colon B\widehat{\UFP}\to B\UFP(1)$ is induced from the inclusion and $a\in\ell^\infty(\Z,\Z)_S^\vee$ is the dual of $(\ldots,1,1,1,1,\ldots)\in\ell^\infty(\Z,\Z)_S$.

  \begin{definition}
    Let $E\to X$ be an $\ell^2(\Z,\C)$-bundle with periodic end, where $X$ is a CW complex. We define
    \[
      \hat{\alpha}_n(E)=f^*(\hat{\alpha}_n)\quad\text{and}\quad\hat{\beta}_n(E)=f^*(\hat{\beta}_n)\in H^*(X;\Q)
    \]
    where $f\colon X\to B\widehat{\UFP}$ is a classifying map of $E$.
  \end{definition}

  By Corollary \ref{inclusion periodic}, $\hat{\alpha}_n(E)$ is essentially the only one non-trivial characteristic class for Hilbert bundles with periodic ends in dimension $2n$. The following is immediate from the definition.

  \begin{proposition}
    Let $E\to X$ be an $\ell^2(\Z,\C)$-bundle with periodic end. For a map $f\colon Y\to X$, we have
      \[
        \hat{\alpha}_n(f^{-1}E)=f^*(\hat{\alpha}_n(E))\quad\text{and}\quad\hat{\beta}_n(f^{-1}E)=f^*(\hat{\beta}_n(E)).
      \]
  \end{proposition}


  \section{The abelian group $\ell^\infty(\Z,\Z)_S$}\label{The abelian group}

  Features of $\ell^2(\Z,\C)$-bundles with ends can be captured by the abelian group $\ell^\infty(\Z,\Z)_S$ through the homotopy equivalence in Theorem \ref{homotopy type UFP} or the characteristic classes $\alpha_n(E;a)$. This section proves two properties of $\ell^\infty(\Z,\Z)_S$ to help understanding $\ell^2(\Z,\C)$-bundles with ends.

  First of all, by Proposition \ref{Q-vector space}, $\ell^\infty(\Z,\Z)_S$ is a $\Q$-vector space. This is a remarkable property because $\ell^\infty(\Z,\Z)$ is not a $\Q$-vector space at all, and played an important role in determining the homotopy type of $\UFP(1)$ as in Theorem \ref{homotopy type UFP}. Next, we prove that an equivalence class of a sequence $a\in\ell^\infty(\Z,\Z)$ in $\ell^\infty(\Z,\Z)_S$ only depends on an asymptotic information of $a$.

  \begin{proposition}
    \label{finite seq}
    Suppose $(a_i)\in\ell^\infty(\Z,\Z)$ satisfies $a_i\ne 0$ for only finitely many $i$. Then $[(a_i)]=0$ in $\ell^\infty(\Z,\Z)_S$.
  \end{proposition}

  \begin{proof}
    For an integer $m$, define a sequence $b(m)=(b_i)\in\ell^\infty(\Z,\Z)$ by
    \[
      b_i=\begin{cases}1&i\le m\\0&i>m.\end{cases}
    \]
    Then $(1-S)(b(m))$ is a sequence whose entries are 0 but the $m$-th entry 1. Since $(a_i)$ in the statement is a finite linear combination of $(1-S)(b(m))$, the proof is done.
  \end{proof}

  Define the natural inclusion
  \[
    \iota\colon\Q\to\ell^\infty(\Z,\Z)_S
    \]
    as the natural extension of
    $\Z\to\ell^\infty(\Z,\Z)_S$ by
$    a\mapsto(\ldots,a,a,a,\ldots)$.
  Then by Corollary \ref{inclusion periodic}, the map $\iota$ is identified with the map $\pi_{2n-1}(\widehat{\UFP})\to\pi_{2n-1}(\UFP(1))$ induced from the inclusion $\widehat{\UFP}\to\UFP(1)$.

  \begin{proposition}
    \label{1/n}
    For a positive integer $n$, $\iota(\frac{1}{n})$ is represented by
    \[
      (\ldots,0,1,\underbrace{0,\ldots,0}_{n-1},1,\underbrace{0,\ldots,0}_{n-1},1,0,\ldots)\in\ell^\infty(\Z,\Z).
    \]
  \end{proposition}

  \begin{proof}
    For $a\in\ell^\infty(\Z,\Z)$, we have $Sa=a+(1-S)(-a)$, implying $[Sa]=[a]$ in $\ell^\infty(\Z,\Z)_S$. Let $c\in\ell^\infty(\Z,\Z)$ be the sequence in the statement. Then
    \[
      c+Sc+S^2c+\cdots+S^{n-1}c=(\ldots,1,1,1,\ldots),
    \]
    implying $n[c]=[\iota(1)]$ in $\ell^\infty(\Z,\Z)_S$. Thus $[\iota(\frac{1}{n})]=[\frac{1}{n}\iota(1)]=[c]$, completing the proof.
  \end{proof}

  We say that $a\in\ell^\infty(\Z,\Z)_S$ is \emph{periodic} if $S^na=S$ for some integer $n$. By Propositions \ref{finite seq} and \ref{1/n}, $\iota(\Q)\subset\ell^\infty(\Z,\Z)_S$ coincides with the subgroup of $\ell^\infty(\Z,\Z)_S$ consisting of periodic elements. Moreover, given a periodic element $(a_i)\in\ell^\infty(\Z,\Z)_S$, the corresponding rational number is given by the average of the period, or equivalently, its Ces\`{a}ro type average
  \[
    \lim_{n\to\infty}\frac{1}{2n+1}\sum_{i=-n}^na_i.
  \]
  Thus we can say that periodicity of unitary operators of finite propagation is detected by periodicity in $\ell^\infty(\Z,\Z)_S$.


  \section{Pushforward}\label{Pushforward}

  Throughout this section, a vector bundle will mean a complex vector bundle of finite rank. Vector bundles may be described by their sheaves of sections. There is the pullback of sheaves, and the pullback of vector bundles correspond to the pullback of sheaves. There is also the pushforward of sheaves, but some conditions are needed for constructing the pushforward of vector bundles which correspond to the pushforward of sheaves. A typical topological context that we can construct the pushforward of vector bundles is a finite covering. See \cite{S}. The pushforward of vector bundles along finite coverings appear in several contexts of geometry and topology, and are of particular interest. For example, Fulton and MacPherson \cite{FM} gave a beautiful formula for characteristic classes of the pushforwards of vector bundles along finite coverings.

  The pushforward of vector bundles along coverings have been considered only for finite coverings because we have been working only with vector bundles of finite rank. However, if we allow infinite dimensional fibers, we may consider infinite coverings. This section introduces the push- forward of vector bundles along infinite coverings, and show Hilbert bundles with ends naturally appear in this context.


  \subsection{Completed sum}

  We begin with the simplest case, the pushforward along a trivial covering. Let $\{E_i\}_{i\in I}$ be a family of vector bundles over a common base $X$. For each $x\in X$, let $\widehat{E}_x$ be the Hilbert space obtained by completing the direct sum $\bigoplus_{i\in I}(E_i)_x$, where $(E_i)_x$ is the fiber of $E_i$ at $x\in X$. Let
  \[
    \widehat{E}=\coprod_{x\in X}\widehat{E}_x
  \]
  which we call the \emph{completed sum} of $\{E_i\}_{i\in I}$. Note that $\widehat{E}$ has not been topologized yet, though each $E_x$ is a topological space. There is a natural map
  \[
    q\colon\widehat{E}\to X,\quad q(E_x)=x.
  \]
  We look for a condition for $\widehat{E}$ being topologized such that $q\colon\widehat{E}\to X$ is a Hilbert bundle. The following example is convincing.

  \begin{example}
    Let $H\to S^2$ denote the Hopf line bundle, and let $\widehat{E}$ be the completed sum of $\{H^{\otimes n}\}_{n\in\Z}$. Consider an open cover of $S^2$ consisting of $S^2$ minus the north pole and $S^2$ minus the south pole. Then this open cover trivialize the family of line bundles $\{H^{\otimes n}\}_{n\in\Z}$ simultaneously, so that the open cover also trivializes $\widehat{E}$. Thus we can give a topology to $\widehat{E}$ by this local trivialization. Observe that the transition function on the equator $\{z\in\C\mid|z|=1\}$ is given by
    \[
      \diag(\cdots,z^{-2},z^{-1},1,z,z^2,\ldots).
    \]
    Then the transition function is not continuous
    in norm topology,
    because $\{z^n\}_{n\in\Z}$ is not equicontinuous. So $\widehat{E}$ is not a Hilbert bundle by the above natural topology.
  \end{example}

  The example above suggests that simultaneous trivialization of a family of vector bundles gives a locally trivializing topology on the completed sum, and equicontinuity of the transition functions of a family of vector bundles may guarantee that the completed sum equipped with the above topology is a Hilbert bundle. We shall prove this naive observation is true. To this end, we need to define equicontinuity precisely.  Let $X,Y$ be metric spaces. Recall that a family of maps $\{f_i\colon X\to Y\}_{i\in I}$ is equicontinuous if for any $\epsilon>0$, there exists $\delta>0$ such that for each $i\in I$
  \[
    d(f_i(x_1),f_i(x_2))<\epsilon\quad\text{whenever}\quad d(x_1,x_2)<\delta.
  \]

  \begin{proposition}
    \label{completed sum}
    Let $\{E_i\}_{i\in I}$ be a family of vector bundles over a common metric space $X$, and let $\widehat{E}$ denote its completed sum. Suppose there is an open cover $X=\bigcup_{\lambda \in \Lambda} \
    U_{\lambda}$ satisfying:
    \begin{enumerate}
      \item each $U_{\lambda}$ trivializes $E_i$ for $i\in I$ simultaneously;

      \item on each $U_{\lambda}\cap U_{\mu}$, the transition functions of $E_i$ for $i\in I$ are equicontinuous.
    \end{enumerate}
    Then $q\colon\widehat{E}\to X$ is a Hilbert bundle.
  \end{proposition}

  \begin{proof}
    Since $U_{\lambda}$ trivializes $E_i$ for $i\in I$ simultaneously, $q^{-1}(U_{\lambda})$ has a natural topology for each $\lambda \in \Lambda$. Then an open cover $\widehat{E}= \bigcup_{\lambda \in \Lambda} \
    q^{-1}(U_{\lambda})$ defines a topology of $\widehat{E}$ such that each $U_{\lambda}$ gives a trivialization of $\widehat{E}$. Then it remains to prove the transition functions are continuous. Let $g_{\lambda \mu}^i$ be the transition function of $E_i$ on $U_{\lambda}\cap U_{\mu}$. Then the transition function of $\widehat{E}$ on
    $U_{\lambda}\cap U_{\mu}$ is
    \begin{equation}
      \label{transition function}
      \bigoplus_{i\in I}g_{\lambda \mu}^i.
    \end{equation}
    Since $g_{\lambda \mu}^i$ for $i\in I$ is equicontinuous, for any $\epsilon>0$ there is $\delta$ such that $|g_{\lambda \mu }^i(x)-g_{\lambda \mu}^i(y)|<\epsilon$ whenever $|x-y|<\delta$ for $x,y\in U_{\lambda}\cap U_{\mu}$. Then it follows that
    \[
      \left|\bigoplus_{i\in I}g_{\lambda \mu}^i(x)-\bigoplus_{i\in I}g_{\lambda \mu}^i(y)\right|\le\sup\{|g_{\lambda \mu}^i(x)-g_{\lambda \mu}^i(y)|\mid i\in I\}<\epsilon
    \]
    whenever $|x-y|<\delta$ for $x,y\in U_{\lambda}\cap U_{\mu}$. Thus the transition functions of $\widehat{E}$ are continuous, completing the proof.
  \end{proof}

  Assuming an mild condition together with those of Proposition \ref{completed sum}, the completed sum turns out to be a Hilbert bundle with end.

  \begin{theorem}
    \label{completed sum end}
    Let $\{E_i\}_{i\in I}$ be a family of vector bundles over a common metric space $X$ satisfying the conditions in Proposition \ref{completed sum}. Suppose the conditions:
    \begin{enumerate}
      \item the index set $I$ is a countable metric space;

      \item $\sup\{\dim E_i\mid i\in I\}<\infty$.
    \end{enumerate}
    Then $\pi\colon\widehat{E}\to X$ is a Hilbert bundle with end.
  \end{theorem}

  \begin{proof}
    By Proposition \ref{completed sum}, $q\colon\widehat{E}\to X$ is a Hilbert bundle, so it remains to show it has an end, or equivalently, the transition functions are continuous maps into $\UFP(H)$, where $H$ is the typical fiber of $q$. The transition functions are of the form \eqref{transition function}. Then by assumption, the transition functions take values in $\UFP(H)$, and we can show that they are continuous too quite similarly to the proof of Proposition \ref{completed sum}.
  \end{proof}

  We show properties of characteristic classes of $\ell^2(\Z,\C)$-bundles by using completed sums. For an integer $k\in\Z$, let $E(k)\to S^{2n}$ be a vector bundle of rank $2n$ classified by a map $g_k\colon S^{2n}\to B\U(2n)$ representing $k\in\Z\cong\pi_{2n}(B\U(2n))$. For $a=(a_i)\in\ell^\infty(\Z,\Z)$, let $E(a)$ denote the completed sum of a family $\{E(a_i)\}_{i\in\Z}$ of vector bundles over $S^{2n}$, which indexed by a metric space $\Z$. By Theorem \ref{completed sum end}, $E(a)\to S^{2n}$ is an $\ell^2(\Z,\C)$-bundle with end.

  \begin{proposition}
    \label{completed sum c_n}
    Given any $\bar{a}\in\ell^\infty(\Z,\Z)^\vee_S$ and $n\ge 1$, there is an $\ell^2(\Z,\C)$-bundle with end $E\to B$ such that $\alpha_n(E;\bar{a})\ne 0$.
  \end{proposition}

  \begin{proof}
    We take  $b=(b_i)\in\ell^\infty(\Z,\Z)$
    such that $\bar{a}(b) \ne  0 $.
     Consider $E(b)\to S^{2n}$. Then its transition function is $[b]\in\ell^\infty(\Z,\Z)_S\cong\pi_{2n-1}(\UFP(1))$. Thus $\alpha_n(E(b);\bar{a})\ne 0$, completing the proof.
  \end{proof}

  We show interesting properties of the bundle $E(a)\to S^{2n}$.

  \begin{proposition}
    \label{E(a) property}
    \begin{enumerate}
      \item If $a,a'\in\ell^\infty(\Z,\Z)$ differs by only finitely many entries, then
      \[
        E(a)\cong E(a')
      \]
      as Hilbert bundles with ends.

      \item If $a=(a_i)\in\ell^\infty(\Z,\Z)$ satisfies $a_i\ne 0$ only finitely many $i$, then $E(a)$ is trivial as a Hilbert bundle with end.
    \end{enumerate}
  \end{proposition}

  \begin{proof}
    As in the proof of Proposition \ref{completed sum c_n}, the transition function of $E(a)$ represents $[a]\in\ell^\infty(\Z,\Z)_S\cong\pi_{2n-1}(\UFP(1))$. Then the statements follows from Proposition \ref{finite seq}.
  \end{proof}


  \subsection{Pushforward}

  We define the pushforward of a vector bundle along an infinite covering. For the rest of this section, let $X$ and $Y$ be connected metric spaces, let $E\to X$ be a vector bundle, and let $\pi\colon X\to Y$ be a covering space with a countable fiber $F$. Notice that each fiber of $\pi$ is a countable metric space.

  For each $y\in Y$, let $F_y$ be the Hilbert space obtained by completing the direct sum $\bigoplus_{x\in\pi^{-1}(y)}E_x$, where $E_x$ is the fiber of $E$ at $x\in X$. Let
  \[
    \pi_*E=\coprod_{y\in Y}F_y
  \]
  which we call the \emph{pushforward} of $E$ along $\pi$. Then we get a natural map
  \[
    q\colon\pi_*E\to Y,\quad q(F_y)=y.
  \]
  By definition, if $\pi$ is a trivial covering, then $\pi_*E$ is the completed sum of vector bundles over components of $X$.

  First, we generalize Proposition \ref{completed sum} to pushforwards. Since being a Hilbert bundle is a local property, we obtain the following generalization by translating the conditions of Proposition \ref{completed sum} into the context of pushforwards.

  \begin{proposition}
    \label{pushforward}
    Suppose there is an open cover
     $Y=\bigcup_{\lambda \in \Lambda}
     \ U_{\lambda}$ satisfying:
    \begin{enumerate}
      \item $\pi^{-1}(U_{\lambda})$ is  trivial for each $i$ such that $\pi^{-1}(U_{\lambda})=\coprod_{x\in F}U_{\lambda}^x$ and
      \[
        \pi\colon U_{\lambda}^x\to U_{\lambda}
      \]
      is an isometry for each $x\in F$;

      \item $E$ is trivial on each $\pi^{-1}(U_{\lambda})$;

      \item for each $\lambda, \mu \in \Lambda$, the transition functions of $E$ on $U_{\lambda}^k\cap U_{\mu}^l$ for $k,l\in\Z$ are equicontinuous.
    \end{enumerate}
    Then $q\colon\pi_*E\to Y$ is a Hilbert bundle.
  \end{proposition}

  \begin{proof}
    The proof is basically the same as that for Proposition \ref{completed sum}. The conditions (1) and (2) guarantee the local triviality of $q$. Let $g_{\lambda \mu}^{xy}$ be the transition function of $E$ on $U_{\lambda}^x\cap U_{\mu}^y$. Then the transition function of $\pi_*E$ on $U_{\lambda}\cap U_{\mu}$ is given by
    \begin{equation}
      \label{transition function pushforward}
      \bigoplus_{x,y\in F}g_{\lambda \mu }^{xy}.
    \end{equation}
    Then by arguing as in the proof of Proposition \ref{completed sum}, we can see that the condition (3) guarantees that the transition functions of $\pi_*E$ is continuous, completing the proof.
  \end{proof}

  \begin{example}
    Let $G$ be a countable group acting freely on a metric space $X$. If the action is isometric, then it is properly discontinuous, implying the projection $X\to X/G$ is a covering. Moreover, this covering satisfies the condition (1) of Proposition \ref{pushforward}. A typical example of such a group action is the canonical action of $\Z$ on $\R$, where the resulting covering is the universal cover $\R\to S^1$.
  \end{example}

  Next, we generalize Theorem \ref{completed sum end} to pushforwards. Under the conditions of Proposition \ref{pushforward}, $q\colon\pi_*E\to Y$ is a Hilbert bundle whose transition functions are of the form \eqref{transition function pushforward}. Then we need to see when the transition functions \eqref{transition function pushforward} take valued in unitary operators of finite propagation. To this end, we carefully look at the trivialization of $\pi\colon X\to Y$. In Proposition \ref{pushforward}, we use the index $\Z$ for expressing the trivialization of $\pi$ because it is a countable covering. However, the index set is actually the fiber $F$ of $\pi$ at the basepoint of $Y$, which is a countable metric space, and an end of each fiber $F_y$ of $q$ by using a structure of $F$ and a path from $y\in Y$ to the basepoint of $Y$. So if the deck transformations behave badly, then the transition functions \eqref{transition function} do not take values in unitary operators of finite propagation.

  \begin{theorem}
    \label{pushforward end}
    Suppose
  the conditions of Proposition \ref{pushforward},
  and suppose further that $Y$ is compact.
  Then,
  $q\colon\pi_*E\to Y$ is a Hilbert bundle with end.
  \end{theorem}

  \begin{proof}
    By Proposition \ref{pushforward}, we only need to show the transition functions are continuous and take values in unitary operators of finite propagation.
    Since $Y$ is compact, there is
    a finite open cover $Y=U_1\cup\cdots\cup U_n$ satisfying the conditions  in Proposition \ref{pushforward}. Then for each $i$
    \[
      \pi^{-1}(U_i)=\coprod_{x\in F}U_i^x
    \]
    such that $\pi\colon U_i^x\to U_i$ is isometric for each $x\in F$. Let $g_{ij}^{xy}$ be the transition function of $E$ on $U_i^x\cap U_j^y$. Then the transition function of $\pi_*E$ on $U_i\cap U_j$ is given by
    \[
      \bigoplus_{x,y\in F}g_{ij}^{xy}.
    \]
    By Proposition \ref{pushforward}, it remains to show $g_{ij}^{xy}=0$ for $d(x,y)>L$ and some $L$, which is equivalent to that the transition functions take values in unitary operators of finite propagation. Then it suffices to show that the deck transformations of $\pi$ are uniformly continuous.
    Since $Y$ is compact,  there is $L \geq 0$ such that
    $\text{diam } U_i \leq \frac{L}{2}$ for each $i$.
    Then, for each $a \in U_i^x \cap U_j^y$,
    $d(x,y) \leq d(x,a)+d(a,y) \leq L$.
    This implies that $g_{ij}^{xy}=0$ for $d(x,y) > L$,
         completing the proof.
  \end{proof}

  We show a fundamental property of pushforwards.

  \begin{proposition}
    \label{transfer}
    Let $\overline{E}\to Y$ be a vector bundle which trivializes simultaneously with $\pi$. Then $\pi_*(\pi^{-1}\overline{E})$ is the completed sum of countably many copies of $\overline{E}$.
  \end{proposition}

  \begin{proof}
    We use the same notation as in the proof of Theorem \ref{pushforward end}. Then $g_{ij}^{xy}=0$ for $x\ne y$, implying that $\pi_*(\pi^{-1}\overline{E})$ is the completed sum of countably many copies of $\overline{E}$.
  \end{proof}

  We express the property of pushforward in Proposition \ref{transfer} by characteristic classes.

  \begin{corollary}
    \label{transfer c_n}
    Let $\overline{E}\to Y$ be a vector bundle which trivializes simultaneously with $\pi$. Then
    \[
      \hat{\alpha}_n(\pi_*(\pi^{-1}\overline{E}))\equiv c_n(
     \overline{E} )\mod(c_1(\overline{E}),\ldots,c_{n-1}(\overline{E})).
    \]
  \end{corollary}

  \begin{proof}
    Clearly, $\pi_*(\pi^{-1}\overline{E})$ is a Hilbert bundle with periodic end, so that
    $\hat{\alpha}_n(\pi_*(\pi^{-1}\overline{E}))$ is well defined. By Proposition \ref{transfer}, $\pi_*(\pi^{-1}\overline{E})$ is classified by the composition of maps
    \[
      Y\xrightarrow{f}B\U(n)\xrightarrow{\Delta}B\widehat{\UFP}
    \]
    where $f$ is a classifying map of $\overline{E}$ and $\Delta$ is induced from the diagonal inclusion $\U(n)\to\widehat{\UFP}$. By \cite{KKT1}, $\Delta$ factors as the composition
    \[
      \U(n)\to K(\Q,2n-1)\to\widehat{\UFP}
    \]
    where the first map is the looping of the rationalization of $c_n$ and the second map is given by Theorem \ref{homotopy type periodic}. If $\epsilon\colon S^{2n}\to B\U(\infty)$ represents $1\in\Z\cong\pi_{2n}(B\U(\infty))$, then $\epsilon^*(c_n)=(n-1)!u$ for a generator $u$ of $H^{2n}(S^{2n})\cong\Z$. Thus by the definition of $\widehat{\alpha}_n$, the proof is completed.
  \end{proof}

  Now we consider the pushforward along an infinite cyclic covering.

  \begin{proposition}
    \label{universal cover S^1}
    Let $\pi\colon\R\to S^1$ be the universal cover. Then $\pi_*(\R\times\C)\to S^1$ is a Hilbert bundle with end such that
    \[
      \beta_1(\pi_*(\R\times\C))\ne 0.
    \]
  \end{proposition}

  \begin{proof}
 We consider a cover of $S^1$ by its upper half and lower half, which obviously gives trivialization of $\R\times\C$. Then by Theorem \ref{pushforward end}, $\pi_*(\R\times\C)$ is a Hilbert bundle with end. Consider the transition function $g\colon S^0\to\UFP(1)$ of the above cover of $S^1$, where $S^0=\{\pm 1\}$ is the equator of $S^1$. We may assume $g(+1)=1\in\UFP(1)$, so that $g(-1)$ is the matrix $(g_{ij})_{i,j\in\Z}$ where $g_{ij}=1$ for $j=i-1$ and $g_{ij}=0$ otherwise. Then $g(-1)$ is the shift operator $S$, and so $g$ represents a generator of $\pi_0(\UFP(1))\cong\Z$ by Proposition \ref{pi_0}. Thus a classifying map of $\pi_*(\R\times\C)$ represents a generator of $\pi_1(B\UFP(1))\cong\pi_0(\UFP(1))\cong\Z$, and so $\beta_1(\pi_*(\R\times\C))\ne 0$, completing the proof.
  \end{proof}

  \begin{corollary}
    \label{beta_1}
    Let $\pi\colon X\to Y$ be a normal infinite cyclic covering satisfying the conditions in Theorem \ref{pushforward end}. Then $\pi_*(X\times\C)\to Y$ is a Hilbert bundle with end such that
    \[
      \beta_1(\pi_*(X\times\C))\ne 0.
    \]
  \end{corollary}

  \begin{proof}
    Since $\pi\colon X\to Y$ is normal, there is a map $f\colon Y\to S^1$ such that $X\cong f^{-1}\R$ and $f$ is surjective in $\pi_1$, where $\bar{\pi}\colon\R\to S^1$ is the universal cover. Then
    \[
      \pi_*(X\times\C)\cong\pi_*f^{-1}(\R\times\C)\cong f^{-1}\pi_*(\R\times\C)
    \]
    Since $f$ is non-trivial in $\pi_1$, it is injective in $H^1$. Thus by Propositions \ref{characteristic class} and \ref{universal cover S^1},
    \[
      \beta_1(\pi_*(X\times\C))=\beta_1(f^{-1}\bar{\pi}_*(\R\times\C))=f^*(\beta_1(\bar{\pi}_*(\R\times\C))\ne 0.
    \]
    Therefore, the proof is done.
  \end{proof}

  As for even dimensional characteristic classes, given any element $a\in\ell^\infty(\Z,\Z)_S$, we can construct quite similarly to Proposition \ref{completed sum c_n} a vector bundle $E\to\R\times S^{2n}$ such that the pushforward $\pi_*E$ along the universal cover $\R\to S^1$ satisfies $\alpha_n(\pi_*E;\bar{a})\ne 0$
  for  $\bar{a}\in\ell^\infty(\Z,\Z)_S^\vee$.


  \section{Fourier transform}\label{Fourier transform}

  This section constructs a Hilbert bundle with end by using Fourier transform. This bundle is an $\ell^2(\Z,\C)$-bundle, and will turn out to have a non-trivial 1-dimensional characteristic class. We also discuss possible higher dimensional analog.


  \subsection{Fourier transform over a circle}

  Let $S^1=\{e^{2\pi\sqrt{-1}t}\mid t\in\R\}$, and let $L^2(S^1)$ denote the Hilbert space of square
   integrable complex valued functions on $S^1$ as above. Consider a differential operator
  \[
    D=\frac{1}{2\pi\sqrt{-1}}\frac{d}{dt}
  \]
  acting on $L^2(S^1)$. Then $D$ yields the spectral decomposition of $L^2(S^1)$ such that the eigenvalues of $D$ are all integers and the eigenspace of an integer $n$ is spanned by $e_n=e^{2n\pi\sqrt{-1}t}$. Notice that $\{e_n\}_{n\in\Z}$ is an orthonormal basis of $L^2(S^1)$, which is indexed by a countable metric space $\Z$. In particular, we get an isometry
  \[
    L^2(S^1)\cong\ell^2(\Z,\C)
  \]
  which is exactly the Fourier transform.

  Using the Fourier transform, we construct a Hilbert bundle with end. For each $w\in S^1$, we define a flat line bundle
  \[
    L_w=([0,1]\times\C)/(0,z)\sim(1,zw)
  \]
  over $S^1$. Let $L^2(L_w)$ denote the Hilbert space of $L^2$ sections of $L_w$. Then the differential operator also acts on $L^2(L_w)$ because $L_w$ is flat, so that $L^2(L_w)$ has the spectral decomposition by $D$. We show that as $w$ ranges over all elements of $S^1$, these spectral decompositions of $L^2(L_w)$ glue together to yield a Hilbert bundle with end over $S^1$.

  Let
  \begin{equation}
    \label{open cover S^1}
    U_1=\{e^{-2\pi\sqrt{-1}t}\in S^1\mid 0<t<\tfrac{3}{4}\}\quad\text{and}\quad U_2=\{e^{-2\pi\sqrt{-1}t}\in S^1\mid \tfrac{1}{2}<t<\tfrac{5}{4}\}.
  \end{equation}
  We get an open cover $S^1=U_1\cup U_2$. Consider a map
  \[
    \Phi_1\colon\coprod_{w\in U_1}L^2(L_w)\to U_1\times L^2(S^1),\quad f\mapsto(w,f_w)
  \]
  where $f_w(z)=e^{2\pi\sqrt{-1}st}f(z)$ for $f\in L^2(L_w)$, $w=e^{2\pi\sqrt{-1}s}$ and $z=e^{2\pi\sqrt{-1}t}$. Then $\Phi_1$ is a bijection such that it restricts to an isometry on each $L^2(L_w)\to\{w\}\times L^2(S^1)$ for each $w\in U_1$. Then we get a topology on $\coprod_{w\in U_1}L^2(L_w)$ such that $\Phi_1$ is a homeomorphism and restricts to an isometry on each $L^2(L_w)$. On the other hand, the map
  \[
    \Phi_2\colon\coprod_{w\in U_2}L^2(L_w)\to U_2\times L^2(S^1),\quad f\mapsto(w,f_w)
  \]
  also defines a topology on $\coprod_{w\in U_2}L^2(L_w)$ such that $\Phi_2$ restricts to an isometry on each $L^2(L_w)$, where $f_w(z)=e^{2\pi\sqrt{-1}st}f(z)$ for $f\in L^2(L_w)$, $w=e^{2\pi\sqrt{-1}s}$ and $z=e^{2\pi\sqrt{-1}t}$.

  Let $\mathcal{L}^1=\coprod_{w\in S^1}L^2(L_w)$. Then there is a map
  \[
    q\colon\mathcal{L}^1\to S^1,\quad q(L^2(L_w))\to w.
  \]
  Note that $q^{-1}(U_i)$ are topologized by $\Phi_i$ for $i=1,2$. Then we get a topology on $\mathcal{L}=q^{-1}(U_1)\cup q^{-1}(U_2)$ such that the natural map $q\colon\mathcal{L}\to S^1$ trivializes on $U_1$ and $U_2$.

  Next, we consider the transition functions of $\widehat{L}$. Let
  \[
    V_1=\{e^{-2\pi\sqrt{-1}t}\in S^1\mid -\tfrac{1}{4}<t<0\}\quad\text{and}\quad V_2=\{e^{-2\pi\sqrt{-1}t}\in S^1\mid \tfrac{1}{2}<t<\tfrac{3}{4}\}.
  \]
  Then $U_1\cap U_2=V_1\sqcup V_2$. By definition,
  \[
    \Phi_2\circ\Phi_1^{-1}(z,f)=
    \begin{cases}
      (z,f)&(z,f)\in V_1\times L^2(S^1)\\
      (z,z^{-1}f)&(z,f)\in V_2\times L^2(S^1).
    \end{cases}
  \]
  In particular, by using the orthonormal basis $\{e_n\}_{n\in\Z}$ of $L^2(S^1)$ obtained by the spectral decomposition of the differential operator $D$, we get
  \[
    \Phi_2\circ\Phi_1\vert_{V_1}=1\quad\text{and}\quad\Phi_2\circ\Phi_1\vert_{V_2}=1\times S
  \]
  where $S$ is the shift operator. Thus we obtain:

  \begin{theorem}
    \label{Fourier transform bundle}
    The map $q\colon\mathcal{L}^1\to S^1$ is a Hilbert bundle with end such that $\beta_1(\mathcal{L}^1)$ is the dual of the Hurewicz image of a generator $[S]$ of $\pi_0(\UFP(1))\cong\pi_1(B\UFP(1))\cong\Z$.
  \end{theorem}

  By Proposition \ref{inclusion Ures}, the characteristic class $\beta_1$ comes from $B\Ures$, and there is no interpretation of this cohomology class in terms of differential operators as above. So Theorem \ref{Fourier transform bundle} would show a new direction for studying $\Ures$. On the other hand, it was proved by Gross, Nesme, Vogts and Werner \cite{GNVW} that $\pi_0(\U(H))\cong\Z$ for every Hilbert space $H$ equipped with a specific basis indexed by a countable metric space. Then the characteristic class $\beta_1(E)$ can be defined for an arbitrary Hilbert bundle with end $E$. It would be of interest to study $\beta_1$ for general Hilbert bundle with ends.


  \subsection{Fourier transform over a torus}

  We consider a two dimensional analog of the above construction using Fourier transform. Let $T^2$ be a 2-dimensional torus $S^1\times S^1$, and let $L^2(T^2)$ denote the Hilbert space of square integrable complex functions on $T$. Then a differential operator
  \[
    D=\frac{1}{2\pi\sqrt{-1}}(\frac{d}{dt_1}+\frac{d}{dt_2})
  \]
  acts on $L^2(T^2)$ such that the spectral decomposition gives an orthonormal basis $\{e_{n_1,n_2}\}_{n_1,n_2\in\Z}$, where $e_{n_1,n_2}=\frac{1}{2\pi\sqrt{-1}}e^{2\pi\sqrt{-1}(n_1t_1+n_2t_2)}$. In particular, we get an isometry
  \[
    L^2(T^2)\xrightarrow{\cong}\ell^2(\Z^2,\C)
  \]
  by using $\{e_{n_1,n_2}\}_{n_1,n_2\in\Z}$, which is exactly Fourier transform over $T^2$.

  For $w=(w_1,w_2)\in T^2$, we define a flat line bundle
  \[
    L_w=(\C\times[0,1]\times[0,1])/\sim
  \]
  over $T^2$, where $(z,0,t)\sim(zw_1,1,t)$ and $(z,t,0)\sim (zw_2,t,1)$. Let $L^2(L_w)$ denote the Hilbert space of $L^2$ sections of $L_w$. Since $L_w$ is flat, the differential operator $D$ acts on $L^2(L_w)$. Let
  \[
    \mathcal{L}^2=\coprod_{w\in T^2}L^2(L_w).
  \]
  Then $D$ acts on $\mathcal{L}$, and there is a natural map $q\colon\mathcal{L}^2\to T^2$ such that $q(L^2(L_w))=w$. Let $S^1=U_1\cup U_2$ be an open cover given in \eqref{open cover S^1}. Then we get an open cover $T^2=\bigcup_{1\le i,j\le 2}U_i\times U_j$, and quite similarly to the one dimensional case, we obtain a bijection
  \[
    \Phi_{ij}\colon\coprod_{w\in U_i\times U_j}L^2(L_w)\to U_i\times U_j\times L^2(T^2)
  \]
  which restricts to an isometry $L^2(L_w)\to L^2(T^2)$. Moreover, the transition functions are given by
  \[
    \begin{pmatrix}
      1 & 0 \\
      0 & 1
    \end{pmatrix},\quad
    \begin{pmatrix}
      1 & 0 \\
      0 & e^{-2\pi\sqrt{-1}t}
    \end{pmatrix},\quad
    \begin{pmatrix}
      e^{-2\pi\sqrt{-1}t} & 0 \\
      0 &  1
    \end{pmatrix},\quad
    \begin{pmatrix}
      e^{-2\pi\sqrt{-1}t} & 0 \\
      0 &  e^{-2\pi\sqrt{-1}t}
    \end{pmatrix}.
  \]
  Thus we obtain:

  \begin{proposition}
    The map $q\colon\mathcal{L}^2\to T^2$ is a Hilbert bundle with end.
  \end{proposition}

  In the one dimensional case, non-triviality of the bundle is proved by using characteristic class. However, the abelian groups $V_0$ and $V_1$ are not explicitly given in Theorem \ref{homotopy type UFP 2}, so we cannot see non-triviality in dimension 2. The above construction suggests that we can construct a Hilbert bundle with end over an $n$-torus for any $n\ge 3$. However, the problem of determining non-triviality of bundles is left open too.


  \section{Harmonic oscillator}\label{Harmonic oscillator}

  This section constructs a Hilbert bundle with end by the spectral decomposition of a 2-dimensional harmonic oscillator. This bundle is an $\ell^2(\Z,\C)$-bundle, and it will turn out that its 2-dimensional characteristic class is non-trivial. We refer to \cite{R} for more details on harmonic oscillators.


  \subsection{1-dimensional harmonic oscillator}

  We recall 1-dimensional harmonic oscillators to help understanding the 2-dimensional case which is our subject. Let $a$ be a positive real number, and let $L^2(\R,\C)$ denote the Hilbert space of square integrable complex functions. The second order self-adjoint differential operator
  \[
    H=-\frac{d^2}{dx^2}+a^2x^2
  \]
  on $L^2(\R,\C)$ is called a \emph{1-dimensional harmonic oscillator}. Note that a harmonic oscillator $H$ is not compact. We consider the spectral decomposition of $L^2(\R,\C)$ with respect to $H$. The first order differential operators
  \[
    A=\frac{d}{dx}+ax
    \quad\text{and}\quad A^*=-\frac{d}{dx}+ax
  \]
  on $L^2(\R,\C)$ are called an \emph{annihilation operator} and a \emph{creation operator}, respectively. These operators are quite useful for the spectral decomposition. It is easy to verify:

  \begin{lemma}
    \label{1D relation}
    The following identities hold:
    \[
      AA^*=H+a,\quad A^*A=H-a,\quad[A,A^*]=2a,\quad[H,A]=-2aA,\quad[H,A^*]=2aA^*.
    \]
  \end{lemma}

  Let $\psi_0(x)=(a^2\pi)^\frac{1}{4}e^{-\frac{ax^2}{2}}\in L^2(\R,\C)$. Then we have
  \[
    H\psi_0=a\psi_0.
  \]
  For $k\ge 1$, we define $\psi_k\in L^2(\R,\C)$ inductively by
  \[
    \psi_k=\frac{1}{\sqrt{2ka}}A^*\psi_{k-1}.
  \]
  Then by Lemma \ref{1D relation} and induction on $k$, we get
  \[
    H\psi_k=\frac{1}{\sqrt{2ka}}HA^*\psi_{k-1}=\frac{1}{\sqrt{2ka}}(A^*H+2aA^*)\psi_{k-1}=\frac{(2k-1)a+2a}{\sqrt{2ka}}A^*\psi_{k-1}=(2k+1)a\psi_k.
  \]
  So $\psi_k$ is created from $\psi_{k-1}$ by $A^*$, and $\psi_{k-1}$ is created from $\psi_k$ by $A$, vice versa. Now we state the spectral decomposition given by Fourier transform.

  \begin{proposition}
    \label{1D isometry}
    There is an isometry
    \[
      \ell^2(\Z_{\ge 0},\C)\to L^2(\R,\C),\quad k\mapsto\psi_k.
    \]
  \end{proposition}

  This spectral decomposition generalizes to higher dimensional harmonic oscillators, and we will consider its structure by using a Hilbert bundle with end in dimension two.


  \subsection{2-dimensional harmonic oscillator}

  We define a 2-dimensional harmonic oscillator.
  Let $L^2(\R,\C^2)$ denote the Hilbert space of square integrable functions from $\R$ to $\C^2$. A self-adjoint non-compact operator
  \[
    H=-\frac{d^2}{dx^2}+x^2
  \]
  on $L^2(\R,\C^2)$ is called a \emph{2-dimensional Harmonic oscillator}. Note that we normalize the operator, that is, we only consider the case $a=1$ in the 1-dimensional case.

  We bring out a hidden parameter of the spectral decomposition of a 2-dimensional harmonic oscillator. Let $\UH(2)$ denote the space of 2-dimensional unitary matrices which are Hermitian. Then $U^2=E_2$ for each $U\in\UH(2)$ and the 2-dimensional identity matrix $E_2$, and so we assume $H$ is parametrized by $\UH(2)$. This parametrization enables us to define an annihilation operator and a creation operator for $U\in\UH(2)$ by
  \[
    A_U=\frac{d}{dx}+Ux\quad\text{and}\quad A_U^*=-\frac{d}{dx}+Ux
  \]
  respectively. Then we have identities among $H,A_U,A_U^*$ similarly to Lemma \ref{1D relation}, which can be easily verified from the definitions.

  \begin{lemma}
    \label{2D relation}
    For each $U\in\UH(2)$, the following identities hold:
    \begin{alignat*}{3}
      &A_UA_U^*=H+U\qquad&&A_U^*A_U=H-U\\
      &[A_U,A_U^*]=2U&&[H,A_U]=-2UA_U\qquad&&[H,A_U^*]=2UA_U^*.
    \end{alignat*}
  \end{lemma}

  The above identities together with the construction in the 1-dimensional case show that the spectral decomposition of a 2-dimensional harmonic oscillator is parametrized by $\UH(2)$. Namely, a 2-dimensional harmonic oscillator $H$ can be considered as an operator on
  \[
    \coprod_{U\in\UH(2)}F_U
  \]
  such that each $F_U$ admits the spectral decomposition using $A_U$ and $A_U^*$, where $F_U=L^2(\R,\C^2)$ for each $U\in\UH(2)$. Notice that $\coprod_{U\in\UH(2)}F_U$ has not been topologized yet, and we are going to topologize it to yields a Hilbert bundle with end. We start with specifying the topological type of the base space $\UH(2)$.

  \begin{lemma}
    \label{UH}
    There is a diffeomorphism
    \[
      \UH(2)\cong S^2\sqcup S^0.
    \]
  \end{lemma}

  \begin{proof}
    It is straightforward to see that $\UH(2)$ is the space of matrices
    \[
      \begin{pmatrix}
        s&z\\\overline{z}&t
      \end{pmatrix}
    \]
    where $s,t\in\R$ and $z\in\C$ such that $s^2+|z|^2=t^2+|z|^2=1$ and $(s+t)z=0$. Then $\UH(2)$ is a union of subspaces $\{s+t=0\}$ and $\{s+t\ne 0\}\cap\{z=0\}$. Clearly, the former is diffeomorphic with $S^2$ and the latter is the two point set $\{E_2,-E_2\}$, so homeomorphic with $S^0$. Moreover, we can see that these subspaces are disjoint by looking at determinants.
  \end{proof}

  Let $B=\{U\in\UH(2)\mid\det U=-1\}$. Then $B$ is a connected component of $\UH(2)$ homeomorphic with $S^2$. Let
  \[
    \mathcal{E}=\coprod_{U\in B}F_U
  \]
  where $F_U=L^2(\R,\C^2)$. Then there is a natural map
  \[
    q\colon\mathcal{E}\to B,\quad q(F_U)=U.
  \]
  We aim to topologize $\mathcal{E}$ so that $q\colon\mathcal{E}\to B$ is a Hilbert bundle with end. As well as the 1-dimensional case, we have:

  \begin{lemma}
    \label{2D creation}
    If $\alpha_0\in L^2(\R,\C^2)$ satisfies $H\alpha_0=\alpha_0$, then
    \[
      H\alpha_k^U=(2kU+E_2)\alpha_k^U
    \]
    for $k\ge 1$, where $\alpha_k^U$ is inductively defined by
    \[
      \alpha_k^U=A_U^*\alpha_{k-1}^U\quad\text{and}\quad\alpha_0^U=\alpha_0.
    \]
  \end{lemma}

  \begin{proof}
    By Lemma \ref{2D relation}, we have
    \begin{align*}
      H\alpha_k^U&=HA_U^*\alpha_{k-1}^U=(A_U^*H+2UA_U^*)\alpha_{k-1}^U=(A_U^*(2(k-1)U+E_2)+2UA_U^*)\alpha_{k-1}^U\\
      &=(2kU+E_2)A_U^*\alpha_{k-1}^U=(2kU+E_2)\alpha_k^U.
    \end{align*}
    Then the statement is proved.
  \end{proof}

  Let
  \[
    D_+=\left\{\begin{pmatrix}s&z\\\overline{z}&-s\end{pmatrix}\in B\,\middle|\,s\ge 0\right\}\quad\text{and}\quad D_-=\left\{\begin{pmatrix}s&z\\\overline{z}&-s\end{pmatrix}\in B\,\middle|\,s\le 0\right\}.
  \]
  Then $D_+$ and $D_-$ are considered as the upper hemisphere and the lower hemisphere of $B\cong S^2$, so that
  \[
    P=\begin{pmatrix}1&0\\0&-1\end{pmatrix}
  \]
  is the north pole of $B$ and $-P$ is the south pole of $B$. Note that every element of $D_+$ is of the form
  \[
    U_z^+=\begin{pmatrix}\sqrt{1-|z|^2}&z\\\overline{z}&-\sqrt{1-|z|^2}\end{pmatrix}
  \]
  for $|z|\le 1$.

  \begin{lemma}
    \label{intertwiner +}
    Define a map
    \[
      \Phi^+\colon D_+\to D_+,\quad U_z^+\mapsto\frac{1}{\sqrt{2}}\begin{pmatrix}
        \sqrt{1+r}&\sqrt{1-r}e^{\sqrt{-1}\theta}\\
        \sqrt{1-r}e^{-\sqrt{-1}\theta}&-\sqrt{1+r}
      \end{pmatrix}
    \]
    where $z=\sqrt{1-r^2}e^{\sqrt{-1}\theta}$ for $0\le r\le 1$. Then $\Phi^+$ is well-defined and continuous, and satisfies
    \[
      A_U^*\Phi^+(U)=\Phi^+(U)A_P^*
    \]
    for each $U\in D_+$.
  \end{lemma}

  \begin{proof}
    Clearly, the map $\Phi^+$ is well-defined and continuous. It is straightforward to verify
    \[
      U_z^+\Phi^+(U_z^+)=\Phi^+(U_z^+)P.
    \]
    Then the proof is done.
  \end{proof}

Let us consider  the spectral decompositions of $F_U$ with respect to $H$ for $U\in D_+$. Let $\psi_0=\pi^\frac{1}{4}e^{-\frac{x^2}{2}}\in L^2(\R,\C)$, and let
  \[
    \psi_0^1=
    \begin{pmatrix}
      \psi_0\\0
    \end{pmatrix}
    \quad\text{and}\quad
    \psi_0^2=
    \begin{pmatrix}
      0\\\psi_0
    \end{pmatrix}.
  \]
  Clearly, $\psi_0^1$ and $\psi_0^2$ belongs to $L^2(\R,\C)$.
  For $U \in D_+$, $k\ge 1$ and $i=1,2$, we define
  \[
    \psi_{k,+}^{i,U}=A^*_U\psi_{k-1,+}^{i,U}\quad\text{where}\quad\psi_{0,+}^{i,U}=\Phi^+(U)\psi_0^i.
  \]
  \begin{remark}
  The trivial complex plane bundle $\UH(2)\times\mathbb{C}^2$ can be decomposed into the direct sum $L\oplus L'$ of line bundles where $L_U,L'_U\subset\{U\}\times\mathbb{C}^2$ are the eigenspaces of $U$ corresponding to the eigenvalues $1,-1$, respectively.
  The operator $A^\ast_U$ on $L^2(\mathbb{R},\mathbb{C}^2)$ is just the direct sum of
  \[
  -\frac{d}{dx}+x
  \quad\text{on $L^2(\mathbb{R},\mathbb{C})\otimes L_U$}\quad\text{and}\quad
  -\frac{d}{dx}-x
  \quad\text{on $L^2(\mathbb{R},\mathbb{C})\otimes L'_U$.}
  \]
  As in the computation in the proof of Theorem \ref{2-dim}, $L$ and $L'$ are isomorphic to the canonical line bundle on $S^2\cong\mathbb{C}P^1$ and its dual.
  \end{remark}
  
  \begin{lemma}
    \label{ground state P}
    For $U \in \UH(2)$, $k\ge 0$ and $i=1,2$,
    \[
      H\psi_{k,+}^{i,U}=(2kU+E_2)\psi_{k,+}^{i,U}.
    \]
  \end{lemma}

  \begin{proof}
    Since $H\psi_0=\psi_0$, the statement follows from Lemma \ref{2D creation}.
  \end{proof}

  For $U,V\in\UH(2)$, we specify the relation between $\psi_k^{i,U}$ and $\psi_k^{i,V}$. 
  For $U\in D_+$, $k\ge 1$ and $i=1,2$, it follows from Lemma \ref{intertwiner +} that
  if $\psi^{i,U}_{k-1,+}=\Phi^+(U)\Phi^+(P)^{-1}\psi^{i,P}_{k-1,+}$ holds, then
  \begin{align*}
    \psi^{i,U}_{k,+}
    &=A_U^*\psi^{i,U}_{k-1,+}\\
    &=A_U^*\Phi^+(U)\Phi^+(P)^{-1}\psi^{i,P}_{k-1,+}
    =\Phi^+(U)\Phi^+(P)^{-1}A_P^\ast\psi^{i,P}_{k-1,+}
    =\Phi^+(U)\Phi^+(P)^{-1}\psi^{i,P}_{k,+}.
  \end{align*}
  Then by induction on $k$, we get
  \begin{equation}
    \label{plus}
    \psi^{i,U}_{k,+}=\Phi^+(U)\Phi^+(P)^{-1}\psi^{i,P}_{k,+}
  \end{equation}
  for $k\ge0$.

  Quite similarly to Proposition \ref{1D isometry}, we have:

  \begin{lemma}
    \label{phi_+}
    There is an isometry
    \[
      \phi^+\colon\ell^2(\Z,\C)\xrightarrow{\cong}L^2(\R,\C^2),\quad e_k\mapsto
      \begin{cases}
        \psi_{k,+}^{1,U}&k\ge 0\\
        \psi_{-k-1,+}^{2,U}&k<0
      \end{cases}
    \]
    where $e_k$ is a sequence $(\ldots,0,0,\overset{k}{1},0,0,\ldots)$.
  \end{lemma}

  Now we are ready to define a topology on $\coprod_{U\in D_+}F_U$.

  \begin{proposition}
    \label{trivialization +}
    $\coprod_{U\in D_+}F_U$ can be topologized so that
    \[
      \bar{\phi}^+\colon D_+\times\ell^2(\Z,\C)\to\coprod_{U\in D_+}F_U,\quad(U,a)\mapsto\phi^+(a)
    \]
    is a homeomorphism and restricts to an isometry $\{U\}\times\ell^2(\Z,\C)\xrightarrow{\cong}F_U$ for each $U\in D_+$.
  \end{proposition}

  \begin{proof}
    This follows at once from \eqref{plus} and Lemma \ref{phi_+}.
  \end{proof}

  We can also get a topology on $\coprod_{U\in D_-}F_U$ similar to that in Corollary \ref{trivialization +}. Since the proofs are quite analogous, we only write results. Every element of $D_-$ is of the form
  \[
    U_z^-=\begin{pmatrix}-\sqrt{1-|z|^2}&z\\\overline{z}&\sqrt{1-|z|^2}\end{pmatrix}\in D_-
  \]
  where $|z|\le 1$.

  \begin{lemma}
    \label{intertwiner -}
    Define a map
    \[
      \Phi^-\colon D_-\to D_-,\quad U_z^-\mapsto\frac{1}{\sqrt{2}}
      \begin{pmatrix}
        -\sqrt{1-r}&\sqrt{1+r}e^{\sqrt{-1}\theta}\\
        \sqrt{1+r}e^{-\sqrt{-1}\theta}&\sqrt{1-r}
      \end{pmatrix}
    \]
    where $z=\sqrt{1-r^2}e^{\sqrt{-1}\theta}$ for $0\le r\le 1$. Then $\Phi^-$ is well-defined and continuous, and satisfies
    \[
      \Phi^-(U)A_U^*=A_{-P}^*\Phi^-(U).
    \]
  \end{lemma}

  Then as above, defining
  \[
    \psi_{k,-}^{i,U}=A^*_U\psi_{k-1,-}^{i,U}\quad\text{where}\quad\psi_{0,-}^{i,U}=\Phi^-(U)\psi_0^i
  \]
  for $k\ge1$ and $i=1,2$, we have
  \begin{equation}
    \label{minus}
    \psi^{i,U}_k=\Phi^-(U)\Phi^-(-P)^{-1}\psi^{i,-P}_{k,-}.
  \end{equation}

  \begin{lemma}
    For each $U\in D_-$, there is an isometry
    \[
      \phi^-\colon\ell^2(\Z,\C)\xrightarrow{\cong}L^2(\R,\C^2),\quad e_k\mapsto
      \begin{cases}
        \psi_{k,-}^{2,U}&k\ge 0\\
        \psi_{-k-1,-}^{1,U}&k<0.
      \end{cases}
    \]
  \end{lemma}

  \begin{proposition}
    \label{trivialization -}
    $\coprod_{U\in D_-}F_U$ can be topologized so that
    \[
      \bar{\phi}^-\colon D_-\times\ell^2(\Z,\C)\to\coprod_{U\in D_-}F_U,\quad(U,a)\mapsto\phi^-(a)
    \]
    is a homeomorphism and restricts to an isometry $\{U\}\times\ell^2(\Z,\C)\xrightarrow{\cong}F_U$ for each $U\in D_-$.
  \end{proposition}

  We are ready to prove:

  \begin{theorem}
    \label{2-dim}
    The spectral decomposition of a 2-dimensional harmonic oscillator provides $\mathcal{E}$ with a Hilbert bundle with end such that
    \[
      \alpha_1(\mathcal{E};a)\ne 0
    \]
    where $a$ is the dual of $[(\ldots,1,1,-1,-1,\ldots)]\in\ell^\infty(\Z,\Z)_S$.
  \end{theorem}

  \begin{proof}
    By Corollaries \ref{trivialization +} and \ref{trivialization -}, $q\colon\mathcal{E}\to B$ can be topologized by the union
    \[
      \mathcal{E}=\bar{\phi}^+(D_+\times\ell^2(\Z,\C))\cup\bar{\phi}^-(D_-\times\ell^2(\Z,\C))
    \]
    so that $\mathcal{E}$ trivializes on each $D_\pm$.
    For $U=\begin{pmatrix}0&z\\z^{-1}&0\end{pmatrix}\in D_+\cap D_-$,
    let us check
    \[
      \psi_{k,-}^{1,U}
      =-z^{-1}\psi_{k,+}^{2,U}
      \quad\text{and}\quad
      \psi_{k,-}^{2,U}
      =z\psi_{k,+}^{1,U}.
    \]
    For $k=0$, this can be seen by
    \begin{align*}
	  &\psi_{0,-}^{1,U}
	  =\Phi^-(U)\psi_0^1
	  =\frac{\psi_0}{\sqrt{2}}\begin{pmatrix}-1\\z^{-1}\end{pmatrix},
	  \quad
	  &&\psi_{0,+}^{2,U}
	  =\Phi^+(U)\psi_{0,+}^{2,P}
	  =\frac{\psi_0}{\sqrt{2}}\begin{pmatrix}z\\-1\end{pmatrix},\\
	  &\psi_{0,-}^{2,U}
	  =\Phi^-(U)\psi_{0,-}^{2,-P}
	  =\frac{\psi_0}{\sqrt{2}}\begin{pmatrix}z\\1\end{pmatrix},
	  \quad
	  &&\psi_{0,+}^{1,U}
	  =\Phi^+(U)\psi_{0,+}^{1,P}
	  =\frac{\psi_0}{\sqrt{2}}\begin{pmatrix}1\\z^{-1}\end{pmatrix}.
	\end{align*}
    The induction step is verified as
    \begin{align*}
      \psi_{k,-}^{1,U}
      =A_U^\ast\psi_{k-1,-}^{1,U}
      =A_U^\ast(-z^{-1}\psi_{k-1,+}^{2,U})
	  =-z^{-1}\psi_{k,+}^{2,U}.
    \end{align*}
    The relation $\psi_{k,-}^{2,U}=z^{-1}\psi_{k,+}^{1,U}$ similarly follows.
    So the transition function is given by
    \begin{equation}
      \label{transition function phi}
      D_+\cap D_-\to\U(\ell^2(\Z,\C)),\quad z\mapsto\diag(\ldots,-z^{-1},-z^{-1},-z^{-1},z,z,z,\ldots)
    \end{equation}
    where we identify $D_+\cap D_-$ with $\{z\in\C\mid|z|=1\}$. Then the transition function is a continuous function into $\UFP(1)$. Thus $q\colon\mathcal{E}\to B$ is a Hilbert bundles with end.

    It remains to compute the characteristic class. By \eqref{transition function phi}, the transition function represents $a\in\ell^\infty(\Z,\Z)_S\cong\pi_1(\UFP(1))$. Suppose $(\ldots,1,1,-1,-1,\ldots)=(1-S)(b)$ for some $b=(b_i)\in\ell^\infty(\Z,\Z)$. Then $\{|b_i|\mid i\in\Z\}$ must be unbounded, so $[a]$ is non-trivial in $\ell^\infty(\Z,\Z)_S$, completing the proof.
  \end{proof}

  The proof of Theorem \ref{2-dim} implies:

  \begin{corollary}
    Let $E_n=H$ for $n\ge 0$ and $E_n=H^*$ for $n<0$, where $H$ is the Hopf line bundle over $S^2$. Let $\widehat{E}$ be the completed sum of $\{E_n\}_{n\in\Z}$. Then
    \[
      \mathcal{E}\cong\widehat{E}.
    \]
    as Hilbert bundles with ends.
  \end{corollary}

  We give an important remark on Theorem \ref{2-dim}. If we instead choose isometries such that
  \[
    \ell^2(\Z,\C)\to L^2(\R,\C^2),\quad e_k\mapsto
    \begin{cases}
      \psi_n^{1,\pm P}&k=2n\\
      \psi_n^{2,\pm P}&k=2n+1
    \end{cases}
  \]
  then $q\colon\mathcal{E}\to B$ is also a Hilbert bundle with end by the corresponding trivialization. With this choice of isometries, the transition function represents
  \[
    [(\ldots,-1,1,-1,1,\ldots)]\in\ell^\infty(\Z,\Z)_S\cong\pi_1(\UFP(1)).
  \]
  However, this element is trivial because $(\ldots,-1,1,-1,1,\ldots)=(1-S)(\ldots,1,0,1,0,\ldots)$. Then with the above choice of trivializations, $\mathcal{E}$ 
   can be a trivial Hilbert bundle with end. This observation shows that a different choice of trivializations possibly gives a different end of the same Hilbert bundle, and characteristic classes can distinguish different ends of the same Hilbert bundle.

\end{document}